\documentclass{amsart}
\usepackage{amssymb}
\usepackage{amsfonts}

\newtheorem{theorem}{Theorem}
\theoremstyle{plain}

\newtheorem{corollary}{Corollary}

\newtheorem{definition}{Definition}
\newtheorem{example}{Example}

\newtheorem{lemma}{Lemma}

\newtheorem{proposition}{Proposition}
\newtheorem{remark}{Remark}

\numberwithin{equation}{section}
\input{tcilatex}

\begin{document}
\title{A Characterization of Constant-ratio Curves in Euclidean 3-space $%
\mathbb{E}^{3}$}
\author{Selin G\"{u}rp\i nar, Kadri Arslan \& G\"{u}nay \"{O}zt\"{u}rk}
\address{Uluda\u{g} University, Department of Mathematics, Bursa, TURKEY}
\address{Kocaeli University, Department of Mathematics, Kocaeli, TURKEY}

\begin{abstract}
A twisted curve in Euclidean $3$-space $\mathbb{E}^{3}$ can be considered as
a curve whose position vector can be written as linear combination of its
Frenet vectors. In the present study we study the twisted curves of constant
ratio in $\mathbb{E}^{3}$ and characterize such curves in terms of their
curvature functions. Further, we obtain some results of $T$-constant and $N$%
-constant type twisted curves in $\mathbb{E}^{3}$. Finally, we give some
examples of equiangular spirals which are constant ratio curves.
\end{abstract}

\maketitle

\section{\textbf{Introduction}}

A curve $x:I\subset \mathbb{R}\rightarrow \mathbb{E}^{3}$ in Euclidean
3-space is called a \textit{twisted curve} if it has nonzero Frenet
curvatures $\kappa _{1}(s)$ and $\kappa _{2}(s)$. From the elementary
differential geometry it is well known that a curve $x(s)$ in $\mathbb{E}%
^{3} $ lies on a sphere if its position vector (denoted also by $x$) lies on
its normal plane at each point. If the position vector $x$ lies on its
rectifying plane then $x(s)$ is called \textit{rectifying curve} \cite{Ch2}.
Rectifying curves characterized by the simple equation%
\begin{equation}
x(s)=\lambda (s)T(s)+\mu (s)N_{2}(s),  \label{a1}
\end{equation}%
where $\lambda (s)$ and $\mu (s)$ are smooth functions and $T(s)$ and $%
N_{2}(s)$ are tangent and binormal vector fields of $x$ respectively \cite%
{Ch2}. In the same paper B. Y. Chen gave a simple characterization of
rectifying curves. In particular it is shown in \cite{CD} that there exists
a simple relation between rectifying curves and centrodes, which play an
important roles in mechanics kinematics as well as in differential geometry
in defining the curves of constant procession. It is also provide that a
twisted curve is congruent to a non constant linear function of $s$ \cite%
{Ch3}. Further, in the Minkowski $3$-space $\mathbb{E}_{1}^{3}$, the
rectifying curves are investigated in (\cite{ET, IB,INP,IN1}). In \cite{IN3}
\ a characterization of the spacelike, the timelike and the null rectifying
curves in the Minkowski $3$-space in terms of centrodes is given. For the
characterization of rectifying curves in three dimensional compact Lee
groups or in dual spaces see \cite{YAC} or \cite{BGOE} respectively.

For a regular curve $x(s)$, the position vector $x$ can be decompose into
its tangential and normal components at each point:%
\begin{equation}
x=x^{T}+x^{N}.  \label{a2}
\end{equation}

A curve $x(s)$ with $\kappa _{1}(s)>0$ is said to be of \textit{constant
ratio} if the ratio $\left \Vert x^{T}\right \Vert :\left \Vert
x^{N}\right
\Vert $ is constant on $x(I)$ where $\left \Vert
x^{T}\right
\Vert $ and $\left \Vert x^{N}\right \Vert $ denote the length
of $x^{T}$ and $x^{N}$, respectively \cite{Ch1}. Clearly a curve $x$ in $%
\mathbb{E}^{3}$ is of constant ratio if and only if $x^{T}=0$ or $%
\left
\Vert x^{T}\right
\Vert :\left \Vert x\right \Vert $ is constant 
\cite{Ch2}. The distance function $\rho =\left \Vert x\right \Vert $
satisfies $\left
\Vert \func{grad}\rho \right \Vert =c$ for some constant $c
$ if and only if we have $\left
\Vert x^{T}\right \Vert =c\left \Vert
x\right \Vert $. In particular, if $\left
\Vert \func{grad}\rho \right
\Vert =c$ then $c\in \lbrack 0,1]$.

A curve in $\mathbb{E}^{n}$ is called $T$\textit{-constant} (resp. $N$%
\textit{-constant}) if the tangential component $x^{T}$ (resp. the normal
component $x^{N}$) of its position vector $x$ is of constant length \cite%
{Ch1}. It is known that a twisted curve in $\mathbb{E}^{3}$ is congruent to
a $N$-constant curve if and only if the ratio $\frac{\kappa _{2}}{\kappa _{1}%
}$ is a non-constant linear function of an arc-length function $s$, i.e., $%
\frac{\kappa _{2}}{\kappa _{1}}(s)=c_{1}s+c_{2}$ for some constants $c_{1}$
and $c_{2}$ with $c_{1}\neq 0$ \cite{Ch1}.

In the present study, we give a generalization of the rectifying curves in
Euclidean $3$-space $\mathbb{E}^{3}$. We consider a \textit{\ }twisted curve
in Euclidean $3$-space $\mathbb{E}^{3}$ whose position vector satisfies the
parametric equation%
\begin{equation}
x(s)=m_{0}(s)T(s)+m_{1}(s)N_{1}(s)+m_{2}(s)N_{2}(s),  \label{a3}
\end{equation}%
for some differentiable functions, $m_{i}(s)$, $0\leq i\leq 2$. If $%
m_{1}(s)=0$ then $x(s)$ becomes a rectifying curve. We characterize the
twisted curves in terms of their curvature functions $m_{i}(s)$ and give the
necessary and sufficient conditions for the twisted curves to become $T$%
-constant or $N$-constant. We give necessary and sufficient conditions for
twisted curves in $\mathbb{E}^{3}$ to become $W$-curves. We also show that
every $N$-constant twisted curve with nonzero constant $\left \Vert
x^{N}\right \Vert $ is a rectifying curve of $\mathbb{E}^{3}.$ Finally, we
give some examples of equiangular spirals which are constant ratio curves.
We give a charecterization of a\ T-constant curve of second kind in $\mathbb{%
E}^{3}$ to become a concho-spiral.

\section{\textbf{Basic Notations}}

Let $x:I\subset \mathbb{R}\rightarrow \mathbb{E}^{3}$ be a unit speed curve
in Euclidean $3$-space $\mathbb{E}^{3}$. Let us denote $T(s)=x^{\prime }(s)$
and call $T(s)$ as a unit tangent vector of $x$ at $s$. We denote the
curvature of $x$ by $\kappa _{1}(s)=\left \Vert x^{\prime \prime
}(s)\right
\Vert $. If $\kappa _{1}(s)\neq 0$, then the unit principal
normal vector $N_{1}(s)$ of the curve $x$ at $s$ is given by $x^{^{\prime
\prime }}(s)=\kappa _{1}(s)N_{1}(s)$. The unit vector $N_{2}(s)=T(s)\times
N_{1}(s)$ is called the unit binormal vector of $x$ at $s$. Then we have the
Serret-Frenet formulae:%
\begin{eqnarray}
T^{\prime }(s) &=&\kappa _{1}(s)N_{1}(s),  \notag \\
N_{1}^{\prime }(s) &=&-\kappa _{1}(s)T(s)+\kappa _{2}(s)N_{2}(s),  \label{b1}
\\
N_{2}^{\prime }(s) &=&-\kappa _{2}(s)N_{1}(s),  \notag
\end{eqnarray}%
where $\kappa _{2}(s)$ is the torsion of the curve $x$ at $s$ (see, \cite{Gl}
and \cite{R}).

If the Frenet curvature $\kappa _{1}(s)$ and torsion $\kappa _{2}(s)$ of $x$
are constant functions then $x$ is called a screw line or a helix \cite{G}.
Since these curves are the traces of 1-parameter family of the groups of
Euclidean transformations then \ F. Klein and \ S. Lie called them \textit{%
W-curves} \cite{KL}.\ It is known that a twisted curve $x$ in $\mathbb{E}%
^{3} $ is called a \textit{general helix} if the ratio $\kappa
_{2}(s)/\kappa _{1}(s)$ is a nonzero constant on the given curve \cite{OAH}.

For a space curve $x:I\subset \mathbb{R}\rightarrow \mathbb{E}^{3}$, the
planes at each point of $x(s)$\ the spanned by $\left \{ T,N_{1}\right \} ,$ 
$\left \{ T,N_{2}\right \} $ and $\left \{ N_{1},N_{2}\right \} $ are known
as the \textit{osculating plane}, the \textit{rectifying plane} and \textit{%
normal plane} respectively. If the position vector $x$ lies on its
rectifying plane then $x(s)$ is called \textit{rectifying curve}. Similarly,
the curve for which the position vector $x$ always lies in its osculating
plane is called \textit{osculating curve}. Finally, $x$ is called \textit{%
normal curve} if its position vector $x$ lies in its normal plane.

From elementary differential geometry it is well known that \ a curve in $%
\mathbb{E}^{3}$ lies in a plane if its position vector lies in its
osculating plane at each point, and lies on a sphere if its position vector
lies in its normal plane at each point \cite{Ch2}.

\section{\textbf{Constant Ratio Curves in }$\mathbb{E}^{3}$}

In the present section we characterize the twisted curves in $\mathbb{E}^{3}$
in terms of their curvatures. Let $x:I\subset \mathbb{R}\rightarrow \mathbb{E%
}^{3}$ be a unit speed twisted curve with curvatures $\kappa _{1}(s)>0$ and $%
\kappa _{2}(s)$. By definition of the position vector of the curve (also
defined by $x$) satisfies the vectorial equation (\ref{a3}), for some
differential functions $m_{i}(s)$, $0\leq i\leq 2$. By taking the derivative
of (\ref{a3}) with respect to arclength parameter $s$ and using the
Serret-Frenet equations (\ref{b1}), we obtain%
\begin{eqnarray}
x^{\prime }(s) &=&(m_{0}^{\prime }(s)-\kappa _{1}(s)m_{1}(s))T(s)  \notag \\
&&+(m_{1}^{\prime }(s)+\kappa _{1}(s)m_{0}(s)-\kappa _{2}(s)m_{2}(s))N_{1}(s)
\label{c1} \\
&&+(m_{2}^{\prime }(s)+\kappa _{2}(s)m_{1}(s))N_{2}(s).  \notag
\end{eqnarray}%
It follows that%
\begin{eqnarray}
m_{0}^{\prime }-\kappa _{1}m_{1} &=&1,  \notag \\
m_{1}^{\prime }+\kappa _{1}m_{0}-\kappa _{2}m_{2} &=&0,  \label{c2} \\
m_{2}^{\prime }+\kappa _{2}m_{1} &=&0.  \notag
\end{eqnarray}

The following result explicitly determines all twisted $W$-curves in $%
\mathbb{E}^{3}.$

\begin{proposition}
Let $x:I\subset \mathbb{R}\rightarrow \mathbb{E}^{3}$ be a twisted curve
with $\kappa _{1}>0$ and let $s$ be its arclength function. If $x$ is a $W$%
-curve of $\mathbb{E}^{3}$ then the position vector $x$ is given by the
curvature functions 
\begin{eqnarray}
m_{0}(s) &=&-\kappa _{1}\left( \frac{c_{3}\sin as-c_{2}\cos as}{a}-bs\right)
+s+c_{0},  \notag \\
m_{1}(s) &=&c_{2}\sin as+c_{3}\cos as-b,  \label{c6} \\
m_{2}(s) &=&\kappa _{2}\left( \frac{c_{3}\sin as-c_{2}\cos as}{a}-bs\right)
+c_{1},  \notag
\end{eqnarray}%
where $c_{i},$ $(0\leq i\leq 3)$ are integral constants and $a=\sqrt{\kappa
_{1}^{2}+\kappa _{2}^{2}}$, $b=\frac{\kappa _{1}}{a^{2}}$ are real constants.
\end{proposition}

\begin{proof}
Let $x$ be a twisted $W$-curve in $\mathbb{E}^{3}$, then by the use of the
equations (\ref{c2}) we get%
\begin{eqnarray}
m_{0}^{\prime } &=&\kappa _{1}m_{1}+1,  \notag \\
m_{1}^{\prime \prime } &=&-(\kappa _{1}^{2}+\kappa _{2}^{2})m_{1}-\kappa
_{1},  \label{c6*} \\
m_{2}^{\prime } &=&-\kappa _{2}m_{1}.  \notag
\end{eqnarray}%
Further, one can show that the system of equations (\ref{c6*}) has a
non-trivial solution (\ref{c6}). Thus, the proposition is proved.
\end{proof}

\begin{definition}
Let $x:I\subset \mathbb{R}\rightarrow \mathbb{E}^{n}$ be a unit speed
regular curve in $\mathbb{E}^{n}$. Then the position vector $x$ can be
decompose into its tangential and normal components at each point:%
\begin{equation*}
x=x^{T}+x^{N}.
\end{equation*}%
if the ratio $\left \Vert x^{T}\right \Vert :\left \Vert x^{N}\right \Vert $
is constant on $x(I)$ then $x$ is said to be of \textit{constant-ratio, or
equavalently }$\left \Vert x^{T}\right \Vert :\left \Vert x\right \Vert =c=$%
constant\textit{\ }\cite{Ch1}.
\end{definition}

For a unit speed regular curve $x$ in $\mathbb{E}^{n},$ the gradient of the
distance function $\rho =\left \Vert x(s)\right \Vert $ is given by 
\begin{equation}
\func{grad}\rho =\frac{d\rho }{ds}x^{\prime }(s)=\frac{<x(s),x^{\prime }(s)>%
}{\left \Vert x(s)\right \Vert }x^{\prime }(s)  \label{d1}
\end{equation}%
where $T$ is the tangent vector field of $x.$

The following results characterize constant-ratio curves.

\begin{theorem}
\cite{Ch7} Let $x:I\subset \mathbb{R}\rightarrow \mathbb{E}^{n}$ be a unit
speed regular curve in $\mathbb{E}^{n}$. Then $x$ is of constant-ratio with $%
\left \Vert x^{T}\right \Vert :\left \Vert x\right \Vert =c$ if and only if $%
\left \Vert \func{grad}\rho \right \Vert =c$ which is constant.

In particular, for a curve of constant-ratio we have $\left \Vert \func{grad}%
\rho \right \Vert =c\leq 1.$
\end{theorem}

\begin{example}
For any real numbers a,c with $\ 0\leq a\leq c<1$, the curve%
\begin{equation*}
x(s)=\left( \sqrt{c^{2}-a^{2}}s\sin \left( \frac{\sqrt{1-c^{2}}}{\sqrt{%
c^{2}-a^{2}}}\ln s\right) ,\sqrt{c^{2}-a^{2}}s\cos \left( \frac{\sqrt{1-c^{2}%
}}{\sqrt{c^{2}-a^{2}}}\ln s\right) ,as\right)
\end{equation*}

in $\mathbb{E}^{3}$ is a unit speed curve satisfying $\left \Vert \func{grad}%
\rho \right \Vert =c$ (see, \cite{Ch7})$.$
\end{example}

\begin{theorem}
\cite{Ch7} Let $x:I\subset \mathbb{R}\rightarrow \mathbb{E}^{n}$ be a unit
speed regular curve in $\mathbb{E}^{n}$. Then $\left \Vert \func{grad}\rho
\right \Vert =c$ holds for a constant c if and only if one of the following
three cases occurs:

(i) $x(I)$ is contained in a hypersphere centered at the origin.

(ii) $x(I)$ is an open portion of a line through the origin.

(iii) $x(s)=csy(s)$, $c\in (0,1),$ where $y=y(u)$ is a unit curve on the
unit sphere of $\mathbb{E}^{n}$ centered at the origin and $u=\frac{\sqrt{%
1-c^{2}}}{c}\ln s$.
\end{theorem}

As a consequence of Theorem $2,$ one can get the following result.

\begin{corollary}
Let $x:I\subset \mathbb{R}\rightarrow \mathbb{E}^{n}$ be a unit speed
regular curve in $\mathbb{E}^{n}$. Then up to a translation of the arc
length function s, we have

i) $\left \Vert \func{grad}\rho \right \Vert =0$ $\Longleftrightarrow $ $%
x(I) $ is contained in a hypersphere centered at the origin.

ii) $\left \Vert \func{grad}\rho \right \Vert =1$ $\Longleftrightarrow $ $%
x(I)$ is an open portion of a line through the origin.

iii) $\left \Vert \func{grad}\rho \right \Vert =c$ $\Longleftrightarrow \rho
=\left \Vert x(s)\right \Vert =cs,$ for $c\in (0,1).$

iv) If $n=2$ and $\left \Vert \func{grad}\rho \right \Vert =c$ for $c\in
(0,1), $ then the curvature of x satisfies 
\begin{equation*}
\kappa ^{2}=\frac{1-c^{2}}{c^{2}\sqrt{s^{2}+b}},
\end{equation*}%
for some real constant $b$.
\end{corollary}

For twisted curves in $\mathbb{E}^{3}$ we obtain the following results.

\begin{proposition}
Let $x:I\subset \mathbb{R}\rightarrow \mathbb{E}^{3}$ be a unit speed
twisted curve in $\mathbb{E}^{3}.$ If $x$ is of constant-ratio then the
position vector of the curve has the parametrization of the form%
\begin{equation*}
x(s)=\left( c^{2}s+cb\right) T(s)+\left( \frac{c^{2}-1}{\kappa _{1}}\right)
N_{1}(s)+\left( \frac{\kappa _{1}c\left( c^{2}+b\right) }{\kappa _{2}}-\frac{%
\left( c^{2}-1\right) \kappa _{1}^{\prime }}{\kappa _{2}\kappa _{1}^{2}}%
\right) N_{2}(s),
\end{equation*}%
for some differentiable functions, $b\in \mathbb{R},c\in \lbrack 0,1).$
\end{proposition}

\begin{proof}
Let x be a regular curve of constant-ratio. Then, from the previous result
the distance function $\rho $ of $x$ satisfies the equality $\rho
=\left
\Vert x(s)\right \Vert =cs+b$ for some differentiable functions, $%
b,c\in \lbrack 0,1).$ Further, using (\ref{d1}) we get%
\begin{equation*}
\left \Vert \func{grad}\rho \right \Vert =\frac{<x(s),x^{\prime }(s)>}{\left
\Vert x(s)\right \Vert }=c.
\end{equation*}%
Since, $x$ is a twisted curve of $\mathbb{E}^{3},$ then it satisfies the
equality (\ref{a3}). So, we get $m_{0}=c^{2}s+cb.$ Hence, substituting this
value into the equations in (\ref{c2}) one can get%
\begin{equation*}
\begin{array}{l}
m_{1}(s)=\frac{c^{2}-1}{\kappa _{1}}, \\ 
m_{2}(s)=\frac{\kappa _{1}\left( c^{2}s+cb\right) }{\kappa _{2}}-\frac{%
\left( c^{2}-1\right) \kappa _{1}^{\prime }}{\kappa _{2}\kappa _{1}^{2}}.%
\end{array}%
\end{equation*}%
Substituting these values into (\ref{a3}), we obtain the desired result.
\end{proof}

\subsection{\textbf{T-constant Twisted Curves in }$\mathbb{E}^{3}$}

\begin{definition}
Let $x:I\subset \mathbb{R}\rightarrow \mathbb{E}^{n}$ be a unit speed curve
in $\mathbb{E}^{n}$. If \ $\left \Vert x^{T}\right \Vert $ is constant then $%
x$ is called a $T$\textit{-constant curve}. For a $T$-constant curve $x$,
either $\left \Vert x^{T}\right \Vert =0$ or $\left \Vert x^{T}\right \Vert
=\lambda $ for some non-zero smooth function $\lambda $ \cite{Ch2}. Further,
a $T$-constant curve $x$ is called first kind if $\left \Vert
x^{T}\right
\Vert =0$, otherwise second kind.
\end{definition}

As a consequence of (\ref{c2}), we get the following result.

\begin{theorem}
Let $x:I\subset \mathbb{R}\rightarrow \mathbb{E}^{3}$ be a unit speed
twisted curve in $\mathbb{E}^{3}$ with the curvatures $\kappa _{1}>0$ and $%
\kappa _{2}\neq 0.$Then $x$ is a $T$-constant curve of first kind, if and
only if%
\begin{equation}
\frac{\kappa _{2}}{\kappa _{1}}-\left( \frac{\kappa _{1}^{\prime }}{\kappa
_{1}^{2}\kappa _{2}}\right) ^{\prime }=0.\text{\ }  \label{c7}
\end{equation}
\end{theorem}

\begin{proof}
Let $x$ be a $T$-constant twisted curve of first kind. Then, from the first
and third equalities in (\ref{c2}) we get $m_{2}=\frac{m_{1}^{\prime }}{%
\kappa _{2}}$ and $m_{2}^{\prime }+m_{1}\kappa _{2}=0.$ Further,
substituting the differentiation of the first equation and $m_{1}=-\frac{1}{%
\kappa _{1}}$ into the first equation we get the result$.$
\end{proof}

\begin{remark}
Any twisted curve satisfying the equality (\ref{c7}) is a spherical curve
lying on a sphere $S^{2}(r)$ of $\mathbb{E}^{3}$. So every $T$-constant
twisted curves of first kind are spherical (see, \cite{TY}).
\end{remark}

By the use of (\ref{c2}) with (\ref{c7}) one can construct the following
examples.

\begin{example}
The twisted curve given with the parametrization 
\begin{equation}
x(s)=-\cos \left( \dint \kappa _{2}ds\right) N_{1}(s)+\sin \left( \dint
\kappa _{2}ds\right) N_{2}(s),  \label{c9}
\end{equation}%
is a $T$-constant twisted curve of first kind .
\end{example}

\begin{example}
The twisted curve given with the curvatures $\kappa _{1}=s$ and $\kappa _{2}=%
\frac{1}{(\ln s+a)s^{2}}$ is a $T$-constant twisted curve of first kind.
\end{example}

As a consequence of (\ref{c2}), we get the following result.

\begin{theorem}
\bigskip Let $x:I\subset \mathbb{R}\rightarrow \mathbb{E}^{3}$ be a twisted
curve in $\mathbb{E}^{3}$. Then $x$ is a $T$-constant curve of second kind
if and only if%
\begin{equation}
\left( \frac{\kappa _{1}^{^{\prime }}+m_{0}\kappa _{1}^{3}}{\kappa
_{1}^{2}\kappa _{2}}\right) ^{^{\prime }}-\frac{\kappa _{2}}{\kappa _{1}}=0,
\label{c10}
\end{equation}%
holds, for some constant function $m_{0}$.
\end{theorem}

\begin{proof}
Suppose that $x$ is a $T$-constant curve of second kind. Then, by the use of
(\ref{c2}) we get%
\begin{equation}
0=m_{2}^{\prime }+m_{1}\kappa _{2},\text{ }m_{2}=\frac{m_{1}^{\prime
}+\kappa _{1}m_{0}}{\kappa _{2}}.  \label{c10*}
\end{equation}%
Further, substituting the differentiation of second equation and using $%
m_{1}=-\frac{1}{\kappa _{1}}$ with first equation, we get the result$.$
\end{proof}

\begin{corollary}
\bigskip Let $x\in \mathbb{E}^{3}$ be a twisted curve in $\mathbb{E}^{3}$.
If $x$ is a $T$-constant of second kind with non-zero constant first
curvature $\kappa _{1}$ then%
\begin{equation}
\kappa _{2}(s)=\mp \frac{\sqrt{a}}{\sqrt{2s+c_{1}a}}\text{, \  \ }
\label{c11}
\end{equation}%
holds, for some constant functions $c_{1}$ and $a=\kappa _{1}^{2}m_{0}$.
\end{corollary}

\begin{proof}
Suppose, first curvature $\kappa _{1}$ is a constant function then by the
use of (\ref{c10}), we get,%
\begin{equation*}
\left( \frac{1}{\kappa _{2}}\right) ^{^{\prime }}m_{0}\kappa _{1}-\frac{%
\kappa _{2}}{\kappa _{1}}=0\text{, \  \ }
\end{equation*}%
which has a non-trivial solution (\ref{c11}).
\end{proof}

For $T$-constant curves of second kind we give the following results;

\begin{theorem}
Let $x\in \mathbb{E}^{3}$ be a $T$-constant twisted curve of second kind.
Then the distance function $\rho =$ $\left \Vert x\right \Vert $ satisfies 
\begin{equation}
\rho =\pm \sqrt{c_{1}s+c_{2}}.  \label{c11*}
\end{equation}%
for some constants $c_{1}=m_{0}$ and $c_{2}.$
\end{theorem}

\begin{proof}
Let $x\in \mathbb{E}^{3}$ be a $T$-constant twisted curve of second kind
then by definition the curvature function $m_{0}(s)$ of $x$ is constant. So
differentiating the squared distance function $\rho ^{2}=$ $\left \langle
x(s),x(s)\right \rangle $ \ and using (\ref{c10}) we get $\rho \rho ^{\prime
}=m_{0}.$ It is an easy calculation to show that, this differential equation
has a nontrivial solution (\ref{c11*}).
\end{proof}

\begin{theorem}
Let $x\in \mathbb{E}^{3}$ be a $T$-constant twisted curve of second kind.
Then $x$ is a general helix of $\mathbb{E}^{3}$ if and only if 
\begin{equation}
\kappa _{1}(s)=\mp \frac{1}{\sqrt{\lambda s^{2}+2c_{2}s+c_{1}}},\text{ \  \  \
\ }  \label{c12}
\end{equation}%
holds, where $\lambda =\frac{\kappa _{2}}{\kappa _{1}}$ is a non-zero
constant and $c_{2}=m_{0}-\lambda c_{1},$ $c_{1}\in \mathbb{R}$.
\end{theorem}

\begin{proof}
Assume that $x$ is a $T$-constant twisted curve of second kind. Then by the
use of (\ref{c10}), we obtain%
\begin{equation}
\kappa _{1}^{\prime }+\kappa _{1}^{3}\left( m_{0}-\lambda ^{2}s-\lambda
c_{1}\right) =0.\text{ }  \label{c13}
\end{equation}%
Consequently, this differential equation has a non-trivial solution(\ref{c12}%
)$,$ where $c_{1},$ $c_{2}=m_{0}-\lambda c_{1}$ are integral constants and $%
\lambda =\frac{\kappa _{2}}{\kappa _{1}}$ a non-zero constant. \ This
completes proof of the corollary.
\end{proof}

\subsection{\textbf{N-constant Twisted Curves in }$\mathbb{E}^{3}$}

\begin{definition}
Let $x:I\subset \mathbb{R}\rightarrow \mathbb{E}^{3}$ be a unit speed curve
in $\mathbb{E}^{3}$. If $\left \Vert x^{N}\right \Vert $ is constant then $x$
is called a $N$-constant curve.For a $N$-constant curve $x$, either $%
\left
\Vert x^{N}\right \Vert =0$ or $\left \Vert x^{N}\right \Vert =\mu $
for some non-zero smooth function $\mu $ \cite{Ch2}. Further, a $N$-constant
curve $x$ is called first kind if $\left \Vert x^{N}\right \Vert =0$,
otherwise second kind.
\end{definition}

So, for a $N$-constant twisted curve $x$ 
\begin{equation}
\left \Vert x^{N}(s)\right \Vert ^{2}=m_{1}^{2}(s)+m_{2}^{2}(s)  \label{c13*}
\end{equation}%
becomes a constant function.

As a consequence of (\ref{a3}) and (\ref{c2}) with (\ref{c13*}) we get the
following result.

\begin{lemma}
Let $x:I\subset \mathbb{R}\rightarrow \mathbb{E}^{3}$ be a unit speed curve
in $\mathbb{E}^{3}$. Then $x$ is a $N$-constant twisted curve if and only if%
\begin{eqnarray}
m_{0}^{\prime } &=&1+\kappa _{1}m_{1},  \notag \\
m_{1}^{\prime } &=&\kappa _{2}m_{2}-\kappa _{1}m_{0},  \label{c14*} \\
m_{2}^{\prime } &=&-\kappa _{2}m_{1},  \notag \\
0 &=&m_{1}m_{1}^{^{\prime }}+m_{2}m_{2}^{^{\prime }},  \notag
\end{eqnarray}%
hold, where $m_{0}(s),m_{1}(s)$ and $m_{2}(s)$ are differentiable functions.
\end{lemma}

For the $N$-constant twisted curves of first kind we give the following
result.

\begin{proposition}
Let $x:I\subset \mathbb{R}\rightarrow \mathbb{E}^{3}$ be a unit speed curve
in $\mathbb{E}^{3}$. Then $x$ is a $N$-constant twisted curve of first kind
if and only if $\ x(I)$ is an open portion of a line through the origin.
\end{proposition}

\begin{proof}
Suppose that $x$ is $N$-constant twisted curve of $\mathbb{E}^{3},$ then the
equality (\ref{c13*}) holds. Further, if $x$ is of first kind then from (\ref%
{c13*}) $m_{1}=m_{2}=0$ which implies that $\kappa _{1}=\kappa _{2}=0.$ So $%
x $ becomes a part of a straight line.
\end{proof}

\begin{definition}
A space curve $x:I\subset \mathbb{R}\rightarrow \mathbb{E}^{3}$ whose
position vector always lies in its rectifying plane is called a rectifying
curve. So, for a rectifying curve $x:I\subset \mathbb{R}\rightarrow \mathbb{E%
}^{3},$ the position vector $x(s)$ satisfies the simple equation%
\begin{equation*}
x(s)=\lambda (s)T(s)+\mu (s)N_{2}(s),
\end{equation*}%
for some differentiable functions $\lambda (s)$ and $\mu (s)$ \cite{Ch2}.
\end{definition}

The following result of B.Y. Chen provides some simple characterizations of
rectifying curves.

\begin{theorem}
\cite{Ch2} Let $x:I\subset \mathbb{R}\rightarrow \mathbb{E}^{3}$ be a
rectifying curve in $\mathbb{E}^{3}$ with $\kappa _{1}>0$ and let $s$ be its
arclength function. Then:

$i).$ The distance function $\rho =$ $\left \Vert x\right \Vert $ satisfies $%
\rho ^{2}$ $=s^{2}+c_{l}s+c_{2}$ for some constants $c_{1}$ and $c_{2}$.

$ii).$ The tangential component of the position vector of the curve is given
by $\left \langle x,T\right \rangle =s+b$ for some constant $b$.

$iii).$ The normal component $x^{N}$ of the position vector of the curve has
constant length and the distance function $\rho $ is nonconstant.

$iv).$ The torsion $\kappa _{2}$ is nonzero, and the binormal component of
the position vector is constant, i.e., $\left \langle x,N_{2}\right \rangle $
is constant.

Conversely, if $x:I\subset \mathbb{R}\rightarrow \mathbb{E}^{3}$ is a curve
with $\kappa _{1}>0$ and if one of $(i),(ii),(iii),$ or $(iv)$ holds, then $%
x $ is a rectifying curve.\qquad
\end{theorem}

The following result of B.Y. Chen provides some simple characterizations of
rectifying curves in terms of the ratio $\frac{\kappa _{2}}{\kappa _{1}}.$

\begin{theorem}
\cite{Ch2} Let x$\in \mathbb{E}^{3}$ be a curve in $\mathbb{E}^{3}$ with $%
\kappa _{1}>0$. Then x is congruent to a rectifying curve if and only if the
ratio of the curvatures of the curve is a nonconstant linear function in
arclength functions, i.e., $\frac{\kappa _{2}}{\kappa _{1}}(s)=c_{1}s+c_{2}$
for some constants c$_{1}$ and $c_{2}$, with $c_{1}\neq 0.$
\end{theorem}

By the use of Lemma 1, we obtain the following result.

\begin{theorem}
Let $x(s)\in \mathbb{E}^{3}$ be a twisted curve in $\mathbb{E}^{3}$ with $%
\kappa _{1}>0$ and $s$ be its arclength function. If $x$ is a $N$-constant
curve of second kind, then the position vector $x$ of the curve has the
parametrization%
\begin{equation}
x(s)=\left( s+\lambda \right) T(s)+\mu N_{2}(s),\text{ \ }\lambda ,\mu \in 
\mathbb{R}\text{.}  \label{c14}
\end{equation}
\end{theorem}

\begin{proof}
Let $x$ be a $N$-constant twisted curve of second kind then the equation (%
\ref{c14*}) holds. So we get $m_{1}\left( m_{1}^{^{\prime }}-\kappa
_{2}m_{2}\right) =0.$ Hence, there are two possible cases; $m_{1}^{^{\prime
}}-\kappa _{2}m_{2}=0$ or $m_{1}=0$. For the first case one can get $\kappa
_{1}=0$, $\kappa _{2}=0$ which implies that $x$ is $N$-constant curve of
first kind. Hence, one can get $m_{1}=0$ and $m_{2}=0$, which contradics the
fact that $\left \Vert x^{N}(s)\right \Vert =\sqrt{m_{1}^{2}(s)+m_{2}^{2}(s)}
$ is nonzero constant. So this case does not occur. For the second case we
get $m_{0}=s+\lambda $, $m_{1}=0$ and $m_{2}=\mu $ for some constant
functions $\lambda $ and $\mu $. This completes the proof of the theorem.
\end{proof}

\begin{corollary}
Let $x\in \mathbb{E}^{3}$ be a $N$-constant twisted curve of second kind in $%
\mathbb{E}^{3}$ with $\kappa _{1}>0$. Then the ratio of the curvatures of
the curve is a nonconstant linear function of arclength functions.
\end{corollary}

\begin{proof}
Let x be a N-constant twisted curve of then the equation $m_{1}^{^{\prime
}}-\kappa _{2}m_{2}=0$ holds. So, substituting the values $m_{0}=s+\lambda $%
, $m_{1}=0$ and $m_{2}=\mu $ into the previous one, we get $\frac{\kappa _{2}%
}{\kappa _{1}}(s)=\frac{s+\lambda }{\mu }$ for some real constants $\lambda $
and $\mu .$
\end{proof}

\section{The Equiangular Spirals}

\begin{definition}
Let $x:I\subset \mathbb{R}\rightarrow \mathbb{E}^{n}$ be a regular curve in $%
\mathbb{E}^{n}.$ If the angle between the position vector field and the
tangent vector field of the curve $x$ is constant (i.e., the angle between $%
x $ and $T$ is constant) then it is called equiangular (\cite{HNV}).
\end{definition}

The equiangular curves in $\mathbb{E}^{n}$ are charterized by the following
result.

\begin{proposition}
Let $x:I\subset \mathbb{R}\rightarrow \mathbb{E}^{n}$ be an equiangular
regular curve in $\mathbb{E}^{n},$ given with arclength parameter$.$ Then $x$
is of constant-angle curve in $\mathbb{E}^{n}.$
\end{proposition}

\begin{proof}
Let $x(I)\subset \mathbb{E}^{n}$ be an equiangular curve in $\mathbb{E}^{n}.$
Then, by definition 
\begin{equation*}
\cos \alpha =\frac{<x(s),T(s)>}{\rho }=\left \Vert \func{grad}\rho \right
\Vert ,
\end{equation*}%
is a constant function.\ So, $x$ becomes a constant-angle curve of $\mathbb{E%
}^{n}$.
\end{proof}

\begin{remark}
In the plane $\mathbb{E}^{2}$ the equiangular spirals have constant angle $%
\alpha $ $\in \left[ 0,\frac{\pi }{2}\right] $ between $x$ and $T.$ It is a
welknown result that the equiangular spirals of $\mathbb{E}^{2}$ are
characterised by the property that their radius of curvature $R=1/\kappa $
is a first degree function of their arclength $s:R=as+b$ for some real
constants $a$ and $b$, (including the straight lines and the circles as the
particular cases of $0$-spirals and $(\pi /2)$-spirals, respectively) (see, 
\cite{HNV}).
\end{remark}

For the twisted equiangular spirals we get the following result.

\begin{proposition}
Let $x:I\subset \mathbb{R}\rightarrow \mathbb{E}^{2}$ be a unit speed\ curve
in $\mathbb{E}^{2}$. If $x$ is an equiangular spiral then the position
vector of x has the paramerization%
\begin{equation}
\begin{array}{l}
\vspace{2mm}m_{0}(s)=c_{1}\cos \varphi (s)+c_{2}\sin \varphi (s)+\frac{%
a(as+b)}{a^{2}+1}, \\ 
\vspace{2mm}m_{1}(s)=c_{1}\sin \varphi (s)-c_{2}\cos \varphi (s)+\frac{as+b}{%
a^{2}+1}.%
\end{array}
\label{d2}
\end{equation}%
where $\varphi (s)=\frac{1}{a}\ln \left( s+\frac{b}{a}\right) $ is a
differentiable function.
\end{proposition}

\begin{proof}
Let $x:I\subset \mathbb{R}\rightarrow \mathbb{E}^{2}$ be a unit speed curve
in $\mathbb{E}^{2}$. Then from (\ref{c2})$,$ $m_{0}^{\prime }-\kappa m_{1}=1$
and\ $m_{1}^{\prime }+\kappa m_{0}=0,$ hold. Further, assume that $x$ is an
equiangular spiral in $\mathbb{E}^{2}.$ Then, substituting $\kappa =\frac{1}{%
as+b}$ into the equations above we obtain a system of differential equations
which has a non-trivial solution (\ref{d2}).
\end{proof}

\begin{definition}
A concho-spiral in $\mathbb{E}^{3}$ is charecterized by the property that
its first and second radii of curvature, $R_{1}=1/\kappa _{1},R_{2}=1/\kappa
_{2}$ (i.e., the inverse of its first and second Serret-Frenet curvatures $%
\kappa _{1}$ and $\kappa _{2}$) are both first degree functions of their
arclenghts (see, \cite{HT}):%
\begin{equation}
\begin{array}{l}
\vspace{2mm}R_{1}=1/\kappa _{1}=a_{1}s+b_{1}, \\ 
\vspace{2mm}R_{2}=1/\kappa _{2}=a_{1}s+b_{1}.%
\end{array}
\label{d3}
\end{equation}
\end{definition}

For the twisted equiangular spirals we get the following results.

\begin{proposition}
Let $x:I\subset \mathbb{R}\rightarrow \mathbb{E}^{3}$ be a unit speed \
T-constant curve of second kind in $\mathbb{E}^{3}$. If $x$ is a
concho-spiral in $\mathbb{E}^{3}$, then the position vector (\ref{a3}) of $x$
has the paramerization%
\begin{equation}
\begin{array}{l}
\vspace{2mm}m_{1}(s)=-(as+b), \\ 
\vspace{2mm}m_{2}(s)=\left( \frac{as+b\ln (as+b)}{c}-\frac{ad\ln (as+b)}{%
c^{2}}\right) +c_{1}, \\ 
m_{0}(s)=\frac{as+b}{cs+d}\left \{ \left( \frac{as+b\ln (as+b)}{c}-\frac{%
ad\ln (as+b)}{c^{2}}\right) +c_{1}\right \} +a(as+b),%
\end{array}
\label{d4}
\end{equation}%
where $a,b,c$ and d are real constants.
\end{proposition}

\begin{proof}
Suppose that $x$ is a concho-spiral in $\mathbb{E}^{3}$ given with the
curvatures $\vspace{2mm}\kappa _{1}=\frac{1}{as+b},\kappa _{2}=\frac{1}{cs+d}
$. If $x$ is $T$-constant curve of second kind. Then, by the use of (\ref{c2}%
) we get%
\begin{eqnarray*}
\vspace{2mm}m_{1}(s) &=&-as+b, \\
m_{2}(s) &=&\int \left( \frac{as+b}{cs+d}\right) ds, \\
m_{0}(s) &=&\frac{\kappa _{2}m_{2}-a}{\kappa _{1}}.
\end{eqnarray*}%
Further, integrating the second equation and using the third one we get the
result.
\end{proof}

\end{document}